\documentclass{amsart}
\usepackage{amssymb}
\usepackage{amsmath}
\usepackage{amsthm}
\usepackage{bbm}
\usepackage{mathrsfs}
\usepackage{amsmath}
\usepackage{tikz,colortbl,xcolor,graphicx,subfigure}
\makeatletter

\newcommand{\Rmnum}[1]{\expandafter\@slowromancap\romannumeral #1@}
\makeatother

\usepackage{xcolor}
\usepackage{amssymb}
\usepackage{amsmath,amsthm}
\usepackage{mathrsfs}

\usepackage{latexsym,euscript,dsfont}
\usepackage{bbm}
\usepackage{colortbl}
\usepackage{cite}
\usepackage{amsmath,amsthm,amsfonts,amssymb}
\usepackage{latexsym}
\usepackage{enumerate}
\usepackage{mathrsfs}
\def\bc{\begin{center}}
\def\ec{\end{center}}
\def\be{\begin{equation}}
\def\ee{\end{equation}}

\newtheorem{lem}{Lemma}[section]

\newtheorem{dfn}[lem]{Definition}
\newtheorem{pro}[lem]{Proposition}

\newtheorem{thm}{Theorem}

\newtheorem{rem}{Remark}
\numberwithin{equation}{section}

\usepackage{colortbl}

\begin{document}
\title[Full Topological entropy of Xiong chaotic sets beyond the specification property]{Full Topological entropy of Xiong chaotic sets beyond the specification property}
\author{Xinyun Zhang}
\address{School of Mathematics and Information Science, Nanchang Hangkong University, Nanchang 330063, P. R. China}
\email{xinyunzhangnc@163.com}
\thanks{This work was supported by initial fund for Doctors (No. EA202407262)}
\maketitle
\begin{abstract}
 In this paper, we show that for a shift dynamical system $(X_{\mathcal{L}}, \sigma)$ with a weakened form of the specification property, there exists a  Xiong chaotic set $C$ of $X_{\mathcal{L}}$ with full topological entropy everywhere (i.e. the intersection of $C$ and arbitrary non-empty open subset of $X_{\mathcal{L}}$ has
full topological entropy). Moreover, $C$ satisfies that for any non-empty subset $A$ and any continuous map $F: A\rightarrow X_{\mathcal{L}},$ there exists an increasing sequence $\{p_{k}\}_{k=1}^{+\infty}$ of positive integers such that $\lim\limits_{k\rightarrow +\infty}\sigma^{p_{k}}(x)=F(x)$ holds for any $x\in A.$
\end{abstract}

\section{introduction}

The notion of chaos was first mathematically formulated by Li and Yorke\cite{Li.T} in 1975, who introduced the concept of a scrambled set and proved that a map with a period-3 orbit is chaotic. Let $\big(X, f\big)$ be a topological dynamical system with metric $d$, a non-empty set $C\subset X$ is said to be chaotic in the sense of Li-Yorke (or scrambled) if
$$\liminf\limits_{n\rightarrow+\infty}d\big(f^{n}(x), f^{n}(y)\big)=0 \ \textup{and}\ \limsup\limits_{n\rightarrow+\infty}d\big(f^{n}(x), f^{n}(y)\big)>0$$
for any two distinct points $x, y\in C.$

In this definition, the erratic property of the orbits of points in $C$ is described by comparing the full processes of the movement of any two distinct points in $C.$ In order to reveal the highly erratic time independence of points in weakly mixing system, Xiong introduced a kind of chaos in \cite{Xiong1} and \cite{Xiong2}, known as Xiong chaos. More precisely, a non-empty set $C\subset X$ is said to be Xiong chaotic if for any subset $A$ of $C$ and any continuous map $F: A\rightarrow X,$ there exists a strictly increasing sequence of positive integers $\{p_{k}\}_{k=1}^{+\infty}$ such that
$$\lim\limits_{k\rightarrow+\infty}f^{p_{k}}\big(x\big)=F\big(x
\big)$$
for any $x\in A.$

By comparing two definitions above, it is clear that every Xiong chaotic set is a chaotic set in the sense of Li-Yorke. Xiong chaos concerns the simultaneous behavior of finitely many points under a common sequence of iterates, and has become a central topic in topological dynamical systems. 

In \cite{Xiong2}, Xiong showed that for a dynamical system $\big(X, f\big)$ with $X$ a separable locally compact metric space containing at least two points, $f$ is weakly mixing if and only if there exists a $F_{\sigma},$ $c$-dense  Xiong chaotic subset $C$ of $X,$ where $c$-dense means that for any non-empty open set $U$ of $X$, the intersection of $C$ and $U$ is uncountable. In \cite{Xiong3}, Xiong studied the Hausdorff dimension theory of Xiong chaotic sets, establishing the existence of Xiong chaotic sets with full Hausdorff dimension everywhere in the full shift over finite symbols, where full Hausdorff dimension everywhere means that the intersection of this set and any non-empty open subset of $X$ has full Hausdorff dimension.  Wu and Tan generalized this result to the full shift over countable symbols in \cite{W.T07chaos}. In \cite{Lau-Shu}, Lau and Shu constructed a Xiong chaotic set with full topological entropy everywhere (can be understood as similar to ``full Hausdorff dimension everywhere'') for positively expansive system with specification property. We say that a compact dynamical system $\big(X, f\big)$ has the specification property if
there exists a single orbit to interpolate between different pieces of orbits, up to a pre-assigned error. Subsequently, Xiao \cite{Xiao} constructed a multiply Xiong chaotic set with full topological entropy everywhere for positively expansive systems with specification property. Recently, Zhong and Chen \cite{Zhong-Chen} introducded the notions of Xiong chaos and strong Xiong chaos in set-valued dynamical systems.


 
In this paper, we study the Xiong chaotic phenomena in  shift dynamical systems $(X_{\mathcal{L}}, \sigma)$ with a weakened form of specification property.

Given a shift space $X,$ we write $\mathcal{L}$ for the language of $X$ --- that is, the collection of all finite words that appear in sequences $x\in X.$ For shift dynamical system, the specification corresponds to the ability to freely concatenate words using connecting words of fixed length. The specification property is a strong hypothesis for dynamical system. There are many shift dynamical system that can be shown to not having the specification property, for example, the $\beta$-shift, since the set of $\beta$ such that $\beta$-shifts has the specification property only has zero Lebesgue measure. But every $\beta$-shifts has a collection of ``good'' words on which specification holds. That motivates this paper.

There has been a resurgence in interest in specification properties recently, due to important contributions by Pfister and Sullivan (almost specification \cite{PS07}, \cite{Yam09}, \cite{Tho12}), and Varandas (non-uniform specification \cite{Var10}). In order to generalize the result of \cite{Lau-Shu} to a wider range of examples, we ask for the specification property to hold only for words taken from a ``good'' subset $\mathcal{G}$ of $\mathcal{L},$ which is another direction to weaken the specification property.

The paper is organized as follows. In section 2, we introduce some preliminaries and state our main result, a condition for Xiong chaos. Section 3 is devoted to showing the proof of the main result.

\section{Preliminaries and main result}
~
\subsection{LANGUAGES FOR SHIFTS}
We start by recalling the relationship between shift spaces and languages. For further background and proofs see \cite{BH86,LM95} and references therein.

Let $P\geq2$ be an integer, we consider the set $\big\{1, 2, \ldots, P\big\}$ and endow it with discrete topology. Let $\big\{1, 2, \ldots, P\big\}^{\mathbb{N}}$ be the collection of all finite words in the symbols $1, 2, \ldots, P.$
Here $\big\{1, 2, \ldots, P\big\}^{\mathbb{N}}$  is equipped with the metric $$d\big(x, y\big)=2^{-\min\{n:\;x_{n}\neq y_{n}\}}.$$

Define $\sigma : \big\{1, 2, \ldots, P\big\}^{\mathbb{N}} \rightarrow \big\{1, 2, \ldots, P\big\}^{\mathbb{N}}$ be the shift map as follows:\ for $x=x_{1}x_{2}\ldots\in \big\{1, 2, \ldots, P\big\}^{\mathbb{N}}$,
$$\sigma\big(x\big)=x_{2}x_{3}\ldots,$$
i.e. the $(n+1)$-th digit of $x$ is exactly the $n$-th digit of $\sigma\big(x\big)$ for all $n\geq1.$
It is clear that $\sigma$ is a continuous map and the pair $\big(\big\{1, 2, \ldots, P\big\}^{\mathbb{N}}, \sigma\big)$ is a compact dynamical system which is known as the full shift.
\begin{dfn}[see \cite{Clim12}]
A (one-sided) language $\mathcal{L}\subset\big\{1, 2, \ldots, P\big\}^{\mathbb{N}}$ is a collection of words such that
\begin{itemize}
\item[(1)]
if $\omega\in\mathcal{L}$ and $v$ is a subword of $\omega,$ then $v\in\mathcal{L};$
\item[(2)]
if $\omega\in\mathcal{L},$ then there exists $v\in\big\{1, 2, \ldots, P\big\}$ such that $\omega v\in\mathcal{L}.$
\end{itemize}
\end{dfn}
Here and hereafter, juxtaposition denotes concatenation---that is, for any $m,n \in \mathbb{N},$ given two words $\omega=\omega_{1}\ldots\omega_{m},$ and $v=v_{1}\ldots v_{n},$ we write $\omega v=\omega_{1}\ldots\omega_{m}v_{1}\ldots v_{n}.$ Sometimes we also use the symbol $``\sqcup"$ to represent the concatenation operation.

Given a one-sided language $\mathcal{L},$ let $X_{\mathcal{L}}\subset\big\{1, 2, \ldots, P\big\}^{\mathbb{N}}$ be the collection of all sequences $x_{1}x_{2} \ldots$ such that for every $1\leq i\leq j<\infty,$
$$x_{i}x_{i+1} \ldots x_{j-1}x_{j}\in\mathcal{L}.$$
Then $X_{\mathcal{L}}$ is a closed $\sigma$-invariant set. And the pair $\big(X_{\mathcal{L}}, \sigma\big)$ is also a dynamical system which is known as the shift dynamical system associated with the language $\mathcal{L}.$

Note that $X_{\mathcal{L}}$ is also equipped with the metric $$d\big(x, y\big)=2^{-\min\{n:\;x_{n}\neq y_{n}\}}.$$
It is obvious that $\big(X_{\mathcal{L}}, \sigma\big)$ is also a compact dynamical system.

For $x\in X_{\mathcal{L}},$ and $m<n,$ we write $x\big([m, n]\big)=x_{m}x_{m+1} \ldots x_{n }$ for the word that appears in positions $m$ through $n,$ and write $\big|x\big|$ for the length of the word. Let $\mathcal{L}_{n}$ denote the collection of all words of length $n$ in $\mathcal{L}.$ Moreover, for given $\omega\in\mathcal{L},$ we define the cylinder determines in $X_{\mathcal{L}}$ as follows:
$$\big[\omega\big]:=\Big\{x\in X_{\mathcal{L}}|\ x_{i}=\omega_{i}\ {\rm for\ all} \ 1\leq i\leq \big|\omega\big|\Big\}.$$
\subsection{TOPOLOGICAL ENTROPY}

Let $\big(X, f\big)$ be a topological dynamical system. Fix any integer $n\geq 1,$ denote the Bowen metric with length $n$ by $$d_{n}\big(x, y\big)=\max\limits_{0\leq i\leq n-1}d\big(f^{i}(x), f^{i}(y)\big).$$ For any two subsets $E, F$ of $X$ and any integer $n\geq 1,$ define the Bowen distance with length $n$ between them by $$d_{n}\big(E, F\big)=\inf\limits_{x\in E,\, y\in F} d_{n}\big(x, y\big)$$ and the $n$-th Bowen diameter by $$\textup{diam}_{d_{n}}\big(E\big) = \sup\limits_{x, y\in E} d_{n}\big(x, y\big).$$ The $\epsilon$-ball in $\big(X, d_{n}\big)$ is denoted by $$B_{n}\big(x, \epsilon\big)=\big\{y\in X\ \big|\ d_{n}(x, y) <\epsilon\big\}.$$

\begin{dfn}[See \cite{Y.B.}]\label{t1}
Let $\big(X, f\big)$ be a dynamical system and $E$ be a subset of $X.$ We define, for $\epsilon>0, s>0,$ and $N>0,$
$$m\big(E; s, N, \epsilon\big)=\inf_{\Gamma}\sum_{i}e^{-sn_{i}},$$
\noindent where $\Gamma=\big\{B_{n_{i}}(x_{i}, \epsilon)\big\}$ is any collection of $n_{i}$-th Bowen balls, with $\min\limits_{i}\ n_{i}> N,$ that covers $Y.$ Note that $m\big(E; s, N, \epsilon\big)$ is not decreasing with respect to $N,$ define
$$m\big(E; s, \epsilon\big)=\lim_{N\rightarrow +\infty}m\big(E; s, N, \epsilon\big).$$
\noindent The critical value $h_{top}\big(f, E; \epsilon\big)$ is defined as
\begin{equation*}
h_{top}\big(f, E; \epsilon\big)=\left\{
\begin{aligned}
& +\infty, \quad\quad\quad\quad\   \text{if $\big\{s\ \big|\ m\big(E; s, \epsilon\big)<+\infty\big\}=\emptyset$} ;\\
& \inf \big\{s\geq0\ |\ m\big(E; s, \epsilon\big)<+\infty\big\},
 \quad\quad  \text{otherwise}.
\end{aligned}\right.
\end{equation*}

\noindent Note that $h_{top}\big(f, E; \epsilon\big)$ is not increasing with respect to $\epsilon,$ then the topological entropy of $E$ is then defined as $$h_{top}\big(f, E\big)=\lim\limits_{\epsilon\rightarrow0}\,h_{top}\big(f, E; \epsilon\big).$$
If $E=X,$ then $h_{top}\big(f, E\big)$ is the topological entropy of the system $\big(X, f\big),$ and we write $h_{top}\big(f\big)$ for brevity.
\end{dfn}

A dynamical system $\big(X, f\big)$ is called positively expansive if for some constant $\eta >0,$ whenever $x\neq y,$ there exists an integer $n=n\big(x, y, \eta\big)\geq 0$ such that $d\big(f^{n}(x), f^{n}(y)\big)\geq \eta.$ We call $\eta$ an expansive constant for $f.$
It is clear that any shift dynamical system is expansive. We state the following lemma, which will be used in the sequel.
\begin{lem}[See \cite{Lau-Shu} Lemma 2.4]\label{l1}
Let $\big(X, f\big)$ be a positively expansive system with $\eta> 0$ being its expansive constant. Then for any $\delta < \frac{\eta}{4}$ and any $\epsilon > 0,$ there exists $N>0$ (which depends on $\delta, \epsilon$) such that
$$d_{n+N}\big(x, y\big)\leq \delta \Rightarrow d_{n}\big(x, y\big)< \epsilon \quad for\ all\ n> 0.$$
\end{lem}

The following gives an alternative definition of topological entropy.

For $\epsilon> 0$ and a positive integer $n,$ we say that a set $E\subset X$ is $\big(n, \epsilon\big)$-separated if for every pair $x, y$ of distinct points in $E,$ implies $d_{n}\big(x,y\big)> \epsilon.$ Let $c_{n}(\epsilon)$ be the largest cardinality of any $\big(n, \epsilon\big)$-separated set. The following proposition shows that topological entropy of a positively expansive system is easy to calculate.

\begin{pro}[See \cite{GTM79}]\label{p1}
Let $\big(X, f\big)$ be a positively expansive system with $\eta$ being its expansive constant. Then for any $\epsilon < \frac{\eta}{4},$ $$h_{top}\big(f\big)=\displaystyle\lim_{n \rightarrow \infty}\frac{1}{n} \log c_{n}\big(\epsilon\big).$$
\end{pro}
~
\subsection{SPECIFICATION PROPERTY}

Literally, there are various properties that go by the name ``specification''. They are all related to the ability to use a single trajectory to approximate arbitrary orbit segments. In the standard definition of specification introduced by Bowen, the time spent transition between orbit segments has a fixed length, independent of the length of the orbit segments. For symbolic spaces, the property corresponds to being able to freely concatenate words using connecting words of fixed length. There are a number of variations on this definition, see \cite{Bow74, DGS76, PS07, Var10}.

\begin{dfn}\label{t2}
We say that a shift dynamical system $\big(X_{\mathcal{L}}, \sigma\big)$  has the specification
property if there exists $t\in \mathbb{N}$ such that for every $k\in \mathbb{N}$ and $\omega^{1}, \ldots, \omega^{k} \in \mathcal{L},$ there are $u^{1}, \ldots, u^{k-1} \in \mathcal{L}$ with $\big|u^{i}\big|\leq t$ for all $1\leq i\leq k-1,$ such that $$\omega^{1} u^{1} \omega^{2} u^{2} \vee\cdots\vee u^{k-1}\omega^{k} \in \mathcal{L},$$
where we use `` $\vee$'' to denote the concatenation of words.
\end{dfn}

In 2012, Climenhaga and Thompson formulated specification properties that apply only to words taken from a
subset of the language. And they call it the (W)-specification.

\begin{dfn}[See \cite{Clim12} Definition 2.1]\label{t2}
A collection of words $\mathcal{G} \subset \mathcal{L}$  has (W)-specification
property if there exists $t\in \mathbb{N}$ such that for any $k\in\mathbb{N}$ and words $\omega^{1}, \ldots, \omega^{k} \in \mathcal{G},$ there are $u^{1}, \ldots, u^{k-1} \in \mathcal{L}$ with $\big|u^{i}\big|\leq t$ for all $1\leq i\leq k-1,$ such that $$x:=\omega^{1} u^{1} \omega^{2} u^{2} \vee\cdots\vee u^{k-1}\omega^{k} \in \mathcal{L}.$$
\end{dfn}

\begin{rem}
We stress that, in our definition, we only ask that $x\in \mathcal{L}$. We do not
require that $x\in \mathcal{G}$.
\end{rem}
~
\subsection{DYNAMICALLY DEFINED MORAN FRACTALS}
Let $Q$ be a subset of $\mathbb{N}^{+\infty}.$ We say a word $i_{1}i_{2}, \ldots, i_{n}$ is $Q$-admissible if $i_{1}i_{2}, \ldots i_{n}i_{n+1}, \ldots \in Q$ for some $i_{n+1}i_{n+2}, \ldots \in \mathbb{N}^{+\infty}.$ Let $Q_{n}$ be the collection of all $Q$-admissible of length $n.$ We assume that $\#\big(Q_{n}\big)$ is finite for each $n,$ where the symbol $``\#\big(\cdot\big)"$ express the cardinality of a set.

\begin{dfn}[see \cite{Lau-Shu}]\label{d4}
Let $\big(X, f\big)$ be a compact dynamical system, a dynamically defined Moran fractal $F$ of $X$ modeled by $Q$ is defined by
$$F=\bigcap_{n=1} ^{\infty} \bigcup_{i_{1}i_{2}\ldots i_{n} \in Q_{n}}\Delta_{i_{1}i_{2}\ldots i_{n}}$$
where $\Delta_{i_{1}i_{2}\ldots i_{n}}$ are closed subsets (called the basic sets) in the $n$-th level which satisfy:

\begin{itemize}
\item[(C1)]
$\Delta_{i_{1}i_{2}\ldots i_{n}j}\subset \Delta_{i_{1}i_{2}\ldots i_{n}}$ for all $i_{1}i_{2}\ldots i_{n} \in Q_{n}$ and $i_{1}i_{2}\ldots i_{n}j \in Q_{n+1};$
\item[(C2)]
 $\lim\limits_{n\rightarrow +\infty} \textup{diam} \big(\Delta_{i_{1}i_{2}\ldots i_{n}}\big)=0;$
 \item[(C3)]
 (dynamical separation condition) there exists $\delta >0$ and $\big\{l_{n}\big\}_{n=1} ^{+\infty}\uparrow +\infty$ such that, for large $n$, $d_{l_{n}}\big(\Delta_{i_{1}i_{2}\ldots i_{n}}, \Delta_{j_{1}j_{2}\ldots j_{n}}\big)\geq \delta$ for any pair of two distinct words $i_{1}i_{2}\ldots i_{n}$ and $ j_{1}j_{2}\ldots j_{n}\in Q_{n}.$
\end{itemize}
\end{dfn}

\noindent For more about the dynamically Moran fractals, See \cite{Y.B.} Ch 5. The set $Q$ in the definition together with $\big\{l_{n}\big\}_{n=1} ^{+\infty},$ provides an estimation of the topological entropy of the limit set.

\begin{pro}[See \cite{Lau-Shu}]\label{p2}
For a sequence of positive integers $\big\{c_{k}\big\}_{k=1}^{+\infty},$ let $F$ be the dynamically defined Moran fractal modeled by $Q=\prod\limits_{k=1} ^{+\infty}\big\{1, 2,\ldots, c_{k}\big\}.$  If\ $\big\{l_{n}\big\}$ in \textup{(C3)} satisfies $$\lim\limits_{n\rightarrow +\infty} \frac{l_{n+1}}{l_{n}}=1$$ and $$\lim_{n\rightarrow +\infty} \sup_{i_{1}i_{2}\ldots i_{n}}\Big\{\textup{diam} _{d_{l_{n}}}\big(\Delta_{i_{1}i_{2}\ldots i_{n}}\big)\Big\}=0,$$
then for any open set $U$ which has non-empty intersection with $F,$ we have
$$h_{top}\big(f, F\cap U\big)=\liminf\limits_{n\rightarrow +\infty} \frac{\log \#(Q_{n})}{l_{n}}.
$$
\end{pro}
~
\subsection{STATEMENT OF RESULT}
We consider language $\mathcal{L}$ admitting a decomposition $\mathcal{L}=\mathcal{G}\mathcal{C}$ --- that is, there are collections of words $\mathcal{G}, \mathcal{C} \subset \mathcal{L}$ such that every word in $\mathcal{L}$ can be expressed as a concatenation of a word from $\mathcal{G},$ and a word from $\mathcal{C}$ in this order.

Given such decomposition, for each $N \in \mathbb{N},$ we define the collections of words $\mathcal{G}\big(N\big)$ as follows:
$$\mathcal{G}\big(N\big):=\big\{\omega u\big|\ \omega\in\mathcal{G}, u\in\mathcal{C},\big|u\big|\leq N\big \}.$$
\noindent It is clear that $\bigcup\limits_{N\geq1}\mathcal{G}\big(N\big)=\mathcal{L}.$

The following theorem is the main result in the paper.

\renewcommand{\thethm}{\Alph{thm}}
\begin{thm}\label{T1}
Let $\big(X_{\mathcal{L}}, \sigma\big)$ be a shift dynamical system with its language $\mathcal{L}$ admitting a decompostion $\mathcal{L}=\mathcal{G}\mathcal{C},$ and suppose that the following conditions are satisfied:
\begin{itemize}
\item[(I)]
 $\mathcal{G}$ has (W)-specification;
\item[(II)]
 There exists $\tau>0$ such that for every $N \in \mathbb{N}$ and  any $\omega\in\mathcal{G}\big(N\big),$ there exists a word $u$ with $\big|u\big|=\tau$ for which $\omega u\in\mathcal{G}.$
\end{itemize}
\noindent Then there exists a Xiong chaotic set with full topological entropy everywhere.
\end{thm}

\begin{rem}
Condition (II) says that every word in $\mathcal{L}$ can be extended to a word in $\mathcal{G}$.
\end{rem}
\section{Proof of Theorem \ref{T1}}
 In this section, we give the proof of Theorem\ref{T1}. The proof proceeds in two stages: first, we build a Moran fractal with full entropy; second, we insert chaotic blocks without altering the entropy.

We proceed inductively with the construction. Let $\tau$ be the constant in condition (II) of Theorem \ref{T1}, $\eta$ be the expansive constant, $\epsilon'$ be a real number with $0<\epsilon'<\frac{\eta}{16},$ and $\big\{n_{k}\big\}_{k=1} ^{+\infty}$ be a strictly increasing sequence of positive integers. For any integer $k\geq1,$ let $E_{k}=\{\mathbf{u}_{1} ^{(k)}, \mathbf{u}_{2} ^{(k)}, \ldots, \mathbf{u}_{c_{k}}^{(k)}\}$ be a maximal
$\big(n_{k}, \epsilon'\big)$-separated set with the cardinality $c_{k}$ in $X_{\mathcal{L}}$.
We set $$\mathcal{E}=\prod\limits_{k=1}^{+\infty} \{1, 2, \ldots, c_{k}\} ^{N_{k}},$$ where the sequence of positive integers $\{N_{k}\}_{k=1} ^{+\infty}$ will be defined later.

\noindent For every integer $k\geq1,$ let $t_{k}$ be the gap size in Definition \ref{t2} with $t_{k}\leq t.$ Now, we require that the sequence $\{n_{k}\}_{k=1}^{+\infty}$ satisfies $$\lim\limits_{k\rightarrow +\infty}\frac{t_{k}+\tau}{n_{k}+t_{k}+\tau}=0,$$ since the sequence $\{t_{k}\}_{k=1} ^{+\infty}$ has been determined by (W)-specification.

\noindent Moreover, we assume that the sequence $\big\{N_{k}\big\}_{k=1} ^{+\infty}$ is increasing and satisfies $$\lim\limits_{k\rightarrow +\infty}\frac{n_{k}+n_{k+1}+t_{k}+t_{k+1}+3\tau}{N_{k}}=0.$$ Let $\Join_{k}$ be an arbitrary word with length $t_{k}$ for any $k\geq1.$ Then for any $k\geq1,$ and every $\boldsymbol{i}$= $\boldsymbol{i}_{1}\boldsymbol{i}_{2}\ldots \boldsymbol{i}_{k}\ldots \in\mathcal{E}$ with $\boldsymbol{i}_{k}=\big(\boldsymbol{i}_{k}(1), \boldsymbol{i}_{k}(2), \ldots, \boldsymbol{i}_{k}(N_{k})\big)\in\big \{1, 2, \ldots, c_{k}\big\}^{N_{k}},$ we set
$$I_{k}\big(\boldsymbol{i}\big)=\big\{\mathbf{u}_{\boldsymbol{i}_{k}(1)}^{(k)}\big[1, n_{k}\big], \mathbf{u}_{\boldsymbol{i}_{k}(2)}^{(k)}\big[1, n_{k}\big],\ldots, \mathbf{u}_{\boldsymbol{i}_{k}(N_{k})}^{(k)}\big[1, n_{k}\big] \big\}.$$

By condition (II) in Theorem \ref{T1}, for any $k\geq1$ and $1\leq j\leq N_{k}$, there exist corresponding words $\gamma_{\boldsymbol{i}_{k}(j)}^{(k)}$ with $\big|\gamma_{\boldsymbol{i}_{k}(j)}^{(k)}\big|= \tau$ such that $$\mathbf{u}_{{\boldsymbol{i}_{k}(j)}}^{
(k)}\big[1, n_{k}\big] \gamma_{\boldsymbol{i}_{k}(j)}^{(k)}\in \mathcal{G}.$$
Then we can glue together the words in $I_{k}\big(\boldsymbol{i}\big)$ for each $k\in \mathbb{N}$ by using (W)-specification.

For any $k\geq1,$ set
$$y_{\boldsymbol{i}_{k}}=\mathbf{u}_{{\boldsymbol{i}_{k}(1)}}^{
(k)}\big[1, {n_{k}}\big] \gamma_{\boldsymbol{i}_{k}(1)}^{(k)}\Join_{k}\mathbf{u}_{{\boldsymbol{i}_{k}(2)}}^{
(k)}\big[1, {n_{k}}\big] \gamma_{\boldsymbol{i}_{k}(2)}^{(k)}\Join_{k}\vee\cdots\vee \Join_{k}\mathbf{u}_{{\boldsymbol{i}_{k}(N_{k})}}^{
(k)}\big[1, {n_{k}}\big] \gamma_{\boldsymbol{i}_{k}(N_{k})}^{(k)}.$$
By (W)-specification, $y_{\boldsymbol{i}_{k}}\in \mathcal{L}.$ Furthermore, by condition (II), for any $y_{\boldsymbol{i}_{k}},$ we can find corresponding words $\gamma_{k}'$ with $\big|\gamma_{k}'\big|=\tau$ such that $y_{\boldsymbol{i}_{k}} \gamma_{k}'\in \mathcal{G}.$

 Let $x_{\boldsymbol{i}_{1}}=y_{\boldsymbol{i}_{1}}\in \mathcal{L},$ and $x_{\boldsymbol{i}_{1}\boldsymbol{i}_{2}}=x_{\boldsymbol{i}_{1}}\gamma_{1}' \\Join_{2} y_{\boldsymbol{i}_{2}}\gamma_{2}'\in \mathcal{L},$ then there exists $\omega_{\boldsymbol{i}_{1}\boldsymbol{i}_{2}}$ with $\big|\omega_{\boldsymbol{i}_{1}\boldsymbol{i}_{2}}\big|=\tau$ such that $x_{\boldsymbol{i}_{1}\boldsymbol{i}_{2}}\omega_{\boldsymbol{i}_{1}\boldsymbol{i}_{2}}
 \in\mathcal{G},$ then define $$x_{\boldsymbol{i}_{1}\boldsymbol{i}_{2}\boldsymbol{i}_{3}}:=x_{\boldsymbol{i}_{1}\boldsymbol{i}_{2}}
\omega_{\boldsymbol{i}_{1}\boldsymbol{i}_{2}}\Join_{3}y_{\boldsymbol{i}_{3}} \gamma_{3}'\in \mathcal{L}.$$ Suppose that we have already constructed $x_{\boldsymbol{i}_{1}\boldsymbol{i}_{2}\ldots \boldsymbol{i}_{k}},$ then there exists $\omega_{\boldsymbol{i}_{1}\boldsymbol{i}_{2}\ldots \boldsymbol{i}_{k}}$ with length $\tau,$ such that $x_{\boldsymbol{i}_{1}\boldsymbol{i}_{2}\ldots \boldsymbol{i}_{k}}\omega_{\boldsymbol{i}_{1}\boldsymbol{i}_{2}\ldots \boldsymbol{i}_{k}}\in \mathcal{G}.$ let $$x_{\boldsymbol{i}_{1}\boldsymbol{i}_{2}\ldots \boldsymbol{i}_{k+1}}:=x_{\boldsymbol{i}_{1}\boldsymbol{i}_{2}\ldots \boldsymbol{i}_{k}}\omega_{\boldsymbol{i}_{1}\boldsymbol{i}_{2}\ldots \boldsymbol{i}_{k}}\Join_{k+1}y_{\boldsymbol{i}_{k+1}}\gamma_{k+1}'\in \mathcal{L}.$$

Hence, we get a Cauchy sequence $\big\{x_{\boldsymbol{i}_{1}\boldsymbol{i}_{2}\ldots \boldsymbol{i}_{k}}\big\}_{k=1}^{+\infty}$ (since every $x_{\boldsymbol{i}_{1}\boldsymbol{i}_{2}\ldots \boldsymbol{i}_{k}}$ is a prefix of $x_{\boldsymbol{i}_{1}\boldsymbol{i}_{2}\ldots \boldsymbol{i}_{k+1}}$ for any $k\geq1),$ by the compactness of $X_{\mathcal{L}},$ there exists $x_{\boldsymbol{i}}\in X_{\mathcal{L}}$ such that $$x_{\boldsymbol{i}}=\lim\limits_{k\rightarrow +\infty}x_{\boldsymbol{i}_{1}\boldsymbol{i}_{2}\ldots \boldsymbol{i}_{k}}.$$ Let
 \begin{equation}\label{eq1}
q_{k}:=\sum_{j=1}^{k}N_{j}\big(n_{j}+t_{j}+\tau\big)-t_{1}+\big(2k-1\big)\tau
\end{equation}
be the length of words before the first appeared $\Join_{k+1}$ for any $k\geq1.$ Define a map $H: \mathcal{E}\rightarrow X_{\mathcal{L}}$ such that $H\big(\boldsymbol{i}\big)=x_{\boldsymbol{i}}.$

\begin{lem}\label{l3}
 $H\big(\mathcal{E}\big)$ is a dynamically defined Moran fractal with full topological entropy.
\end{lem}

\begin{proof}
Let~$\big\{\epsilon_{k}\big\}_{k\ge1}$ be a strictly decreasing sequence of positive numbers. Pick up a positive integer $h$ such that for any $k\geq2$, one has
$$\epsilon_{1}=\sum_{j=h}^{+\infty}
\frac{1}{2^{q_{j}+1}}<\frac{1}{16}\epsilon'\quad \text{and}\quad\epsilon_{k}=\epsilon_{1}-\sum_{j=1}^{k-1}
\frac{1}{2^{q_{j}+1}}>0.$$

\noindent Fix $i_{1}i_{2}\ldots i_{n}\in \mathcal{E}_{n}$ with $n\geq N_{1}.$ If there exist some $k\geq1$ such that $i_{1}i_{2}\ldots i_{n}=\boldsymbol{i}_{1}\boldsymbol{i}_{2}\ldots \boldsymbol{i}_{k},$ then we set
$$l_{n}=q_{k}\ \textup{and}\ \Delta_{i_{1}i_{2}\ldots i_{n}}=\overline{B_{q_{k}}\big(x_{\boldsymbol{i}_{1}\boldsymbol{i}_{2}\ldots \boldsymbol{i}_{k}}, \epsilon_{k}\big)}.$$

\noindent Otherwise, that is, there exist some $k\geq1$ such that $\boldsymbol{i}_{1}\boldsymbol{i}_{2}\ldots \boldsymbol{i}_{k}$ is a prefix of $i_{1}i_{2}\ldots i_{n}$ and at the meanwhile $i_{1}i_{2}\ldots i_{n}$ is a prefix of $\boldsymbol{i}_{1}\boldsymbol{i}_{2}\ldots \boldsymbol{i}_{k+1},$ then set
$$l_{n}=q_{k}+\big(n-\sum_{j=1}^{k}N_{j}\big)
\cdot\big(n_{k+1}+t_{k+1}+\tau\big)$$
and
$$
\Delta_{i_{1}i_{2}\ldots i_{n}}=\bigcup_{\mbox{\tiny$\begin{array}{c}
\boldsymbol{i}_{1}\boldsymbol{i}_{2}\ldots \boldsymbol{i}_{k}\textup{is a prefix of}\ i_{1}i_{2}\ldots i_{n} \textup{and}\\
i_{1}i_{2}\ldots i_{n}\textup{is a prefix of}\ \boldsymbol{i}_{1}\boldsymbol{i}_{2}\ldots \boldsymbol{i}_{k+1}
\end{array}$}}\overline{B_{q_{k+1}}\big(x_{\boldsymbol{i}_{1}\boldsymbol{i}_{2}\ldots \boldsymbol{i}_{k+1}}, \epsilon_{k+1}\big)}.
$$
We claim that $\bigcap\limits_{n=1}^{+\infty}\bigcup\limits_{i_{1}i_{2}\ldots i_{n}\in \mathcal{E}_{n}}\Delta_{i_{1}i_{2}\ldots i_{n}}$ is a dynamically defined Moran fractal.
The condintion (C2) in Definition \ref{d4} is obviously satisfied since $\epsilon_{k}$ is decreasing to $0$ as $k\rightarrow\infty$. Note that
$$d_{q_{k}}(x_{\boldsymbol{i}_{1}\boldsymbol{i}_{2}\ldots \boldsymbol{i}_{k}}, x_{\boldsymbol{i}_{1}\boldsymbol{i}_{2}\ldots \boldsymbol{i}_{k+1}})=
\frac{1}{2^{q_{k}+1}}
=\epsilon_{k}-\epsilon_{k+1},$$

\noindent then

$$\overline{B_{q_{k+1}}\big(x_{\boldsymbol{i}_{1}\boldsymbol{i}_{2}\ldots \boldsymbol{i}_{k+1}}, \epsilon_{k+1}\big)}\subset\overline{B_{q_{k}}\big(x_{\boldsymbol{i}_{1}\boldsymbol{i}_{2}\ldots \boldsymbol{i}_{k}}, \epsilon_{k}\big)}.$$
Thus, (C1) is satisfied.
~\\
\noindent Since the sequence of closed balls $\big\{\overline{B_{q_{k}}\big(x_{\boldsymbol{i}_{1}\boldsymbol{i}_{2}\ldots \boldsymbol{i}_{k}}, \epsilon_{k}\big)}\big\}_{k=1}^{+\infty}$ is strictly decreasing to $x_{\boldsymbol{i}},$ it is obvious that $H(\mathcal{E})$ is exactly the set $$\bigcap\limits_{n=1}^{+\infty}\bigcup\limits_{{i_{1}i_{2}\ldots i_{n}\in\mathcal{E}_{n}}} \Delta_{i_{1}i_{2}\ldots i_{n}}.$$ Now, we show that the condition (C3) holds. For every pair $i_{1}i_{2}\ldots i_{n},~ j_{1}j_{2}\ldots j_{n}$ of distinct words in $\in \mathcal{E}_{n},$ assume that $t$ is the least number such that $i_{t}\neq j_{t}$ with $1\leq t\leq n.$ Then either $k=0$ such that $1\leq t\leq N_{1}$ or there exists $k\geq1$ such that $N_{1}+N_{1}+\cdots +N_{k}<t\leq N_{1}+N_{1}+\cdots +N_{k+1}.$ In brief, there exist some $k\geq0$ and $1\leq m\leq N_{k+1}$ such that $\mathbf{u}_{\boldsymbol{i}_{k+1}(m)}^{(k+1)}\neq \mathbf{u}_{\boldsymbol{j}_{k+1}(m)}^{(k+1)}.$ Then according to the construction of $H\big(\mathcal{E}\big),$ there exist some integer $1\leq b\leq q_{k+1}$ such that

$$\sigma^{b}\big(x_{\boldsymbol{i}}\big)\big[1, n_{k+1}\big]=\mathbf{u}_{\boldsymbol{i}_{k+1}(m)}^{(k+1)}\big[1,n_{k+1}\big]\ \textup
{and}\ \sigma^{b}\big(x_{\boldsymbol{j}}\big)\big[1, n_{k+1}\big]=\mathbf{u}_{\boldsymbol{j}_{k+1}(m)}^{(k+1)}\big[1,n_{k+1}\big].$$
Thus, we have
$$\begin{aligned}
d_{l_{n}}\big(x_{\boldsymbol{i}}, x_{\boldsymbol{j}}\big)&\geq d_{n_{k+1}}\big(\mathbf{u}_{\boldsymbol{i}_{k+1}(m)}^{(k+1)}, \mathbf{u}_{\boldsymbol{j}_{k+1}(m)}^{(k+1)}\big)-d_{n_{k+1}}\big(\sigma^{b}(x_{\boldsymbol{i}}),
\mathbf{u}_{\boldsymbol{i}_{k+1}(m)}^{(k+1)}\big)-d_{n_{k+1}}\big(\sigma^{b}(x_{\boldsymbol{j}}),
\mathbf{u}_{\boldsymbol{j}_{k+1}(m)}^{(k+1)}\big)\\
&\geq \epsilon'
\end{aligned}
$$
and
$$\begin{aligned}
d_{l_{n}}\big(\Delta_{i_{1}i_{2}\ldots i_{n}}, \Delta_{j_{1}j_{2}\ldots j_{n}}\big)\geq
d_{l_{n}}\big(x_{\boldsymbol{i}}, x_{\boldsymbol{j}}\big)-2\epsilon_{k+1}\geq \epsilon'-2\epsilon_{k+1}>\frac{3}{4}\epsilon'.
\end{aligned}
$$
Moreover, for any $n\geq1,$ $l_{n+1}$ is at most $n_{k+1}+t_{k+1}+\tau$ or $n_{k}+t_{k}+2\tau$ larger than $l_{n}.$ We have
$$\begin{aligned}
\frac{n_{k+1}+t_{k+1}+n_{k}+t_{k}+3\tau}{l_{n}}+1<\frac{n_{k+1}+t_{k+1}+n_{k}+t_{k}+3\tau}{N_{k}}+1.
\end{aligned}
$$
Hence,
\begin{equation}\label{lnln1}
\lim\limits_{n\rightarrow+\infty}\frac{l_{n+1}}{l_{n}}=1.
\end{equation}
Since the set $H\big(\mathcal{E}\big)$ can be viewed as a dynamically defined Moran fractal for the compact dynamical system $\big(X_{\mathcal{L}}, \sigma\big),$ by Proposition \ref{p2}, for any open set $U$ which has non-empty intersection with $H\big(\mathcal{E}\big),$ we have $$h_{top}\big(\sigma, H(\mathcal{E})\cap U\big)=\liminf\limits_{n\rightarrow+\infty}
\frac{\log{\#(\mathcal{E}_{n})}}{l_{n}}.$$
For any $n\geq1,$ there exist some intege $r\geq0$ such that $N_{1}+N_{2}+\cdots +N_{r}\leq n\leq N_{1}+N_{2}+\cdots +N_{r+1}.$ It can be calculated that
$$\begin{aligned}
&\quad\ \frac{\log{\#\big(\mathcal{E}_{n}\big)}}{l_{n}}>\frac{\log{\#\big(\mathcal{E}_{N_{1}+N_{2}+\cdots +N_{r}}\big)}}{l_{N_{1}+N_{2}+\cdots +N_{r+1}}}\\
&=\frac{\log{\#\big(\mathcal{E}_{N_{1}+N_{2}+\cdots +N_{r}}\big)}}{l_{r-2}}\cdot\frac{l_{r-2}}{l_{N_{1}+N_{2}+\cdots +N_{r+1}}} \\
&=\frac{N_{1}\log{c_{1}}+N_{2}\log{c_{2}}+\cdots +N_{r}\log{c_{r}}}{l_{_{r-2}}}\cdot\frac{l_{_{r-2}}}{l_{N_{1}+N_{2}+\cdots +N_{r+1}}}.
\end{aligned}
$$
By (\ref{eq1}), we have

$$l_{_{r-2}}\leq q_{_{r-1}}\leq \sum\limits_{j=1}^{r}N_{j}(n_{j}+t_{j}+\tau)+(2r-1)\tau.$$
Then
$$\liminf\limits_{n\rightarrow+\infty}\frac{\log\#\big(\mathcal{E}_{n}\big)}{l_{n}}
\geq\lim\limits_{r\rightarrow+\infty}\frac{N_{1}\log{c_{1}}+N_{2}\log{c_{2}}+\cdots +N_{r}\log{c_{r}}}{\sum\limits_{j=1}^{r}N_{j}\big(n_{j}+t_{j}+\tau\big)+\big(2r-1\big)\tau}
=\lim\limits_{r\rightarrow+\infty}\frac{\log{c_{r}}}{n_{r}+t_{r}+\tau}.$$
Recall that $\lim\limits_{k\rightarrow +\infty}\frac{t_{k}+\tau}{n_{k}+t_{k}+\tau}=0$, by Proposition \ref{p2}, one has
$$h_{top}\big(\sigma, H\big(\mathcal{E}\big)\cap U\big)=\liminf\limits_{n\rightarrow+\infty}\frac{\log{\#\big(\mathcal{E}_{n}\big)}}{l_{n}}
\geq\lim\limits_{k\rightarrow+\infty}
\frac{\log{c_{k}}}{n_{k}}=h_{top}\big(\sigma\big),$$
where $h_{top}\big(\sigma\big)$ is the topological entropy of $\big(X_{\mathcal{L}}, \sigma\big)$.
\end{proof}

Next, we will insert the chaotic parts into the dynamically defined Moran fractal $H\big(\mathcal{E}\big)$ and maintian the entropy.

\begin{lem}\label{L4}
Let $\big(X_{\mathcal{L}}, \sigma\big)$ be the shift dynamical system which satisfies the conditions of Theorem \ref{T1}, then there exists a Xiong chaotic set with full topological entropy.
\end{lem}

\begin{proof}
We first fix some notations. Let $$\mathcal{E}^{(k)}=\prod\limits_{j=1}^{k}\big\{1, 2, \ldots, c_{j}\big\}^{N_{j}}.$$ 
We denote by $\varphi_{1}^{(k)}, \varphi_{2}^{(k)}, \ldots, \varphi_{s_{k}}^{(k)}$ (with $s_{k}=c_{k}^{\#(\mathcal{E}^{(k)})}$) the maps from $\mathcal{E}^{(k)}$ to $E_{k}=\big\{\mathbf{u}_{1}^{(k)}, \mathbf{u}_{2}^{(k)}, \ldots, \mathbf{u}_{c_{k}}^{(k)}\big\}.$ Fix $\boldsymbol{i}$= $\boldsymbol{i}_{1}\boldsymbol{i}_{2}\ldots \boldsymbol{i}_{k}\ldots \in \mathcal{E}$ with $\boldsymbol{i}_{k}=\big(\boldsymbol{i}_{k}(1), \boldsymbol{i}_{k}(2), \ldots, \boldsymbol{i}_{k}(N_{k})\big)\in \big\{1, 2, \ldots, c_{k}\big\}^{N_{k}}.$ Set

$$J_{1}\big(\boldsymbol{i}\big):=\Big\{\varphi_{1}^{(1)}\big(\boldsymbol{i}_{1}\big)\big[1, n_{1}\big], \varphi_{2}^{(1)}\big(\boldsymbol{i}_{1}\big)\big[1, n_{1}\big], \ldots, \varphi_{s_{1}}^{(1)}\big(\boldsymbol{i}_{1}\big)\big[1, n_{1}\big]\Big\}$$
and
$$J_{k}\big(\boldsymbol{i}\big):=\Big\{\varphi_{1}^{(k)}\big(\boldsymbol{i}_{1}\boldsymbol{i}_{2}\ldots \boldsymbol{i}_{k}\big)\big[1, n_{k}\big], \varphi_{2}^{(k)}\big(\boldsymbol{i}_{1}\boldsymbol{i}_{2}\ldots \boldsymbol{i}_{k}\big)\big[1, n_{k}\big], \ldots, \varphi_{s_{k}}^{(k)}\big(\boldsymbol{i}_{1}\boldsymbol{i}_{2}\ldots \boldsymbol{i}_{k}\big)\big[1, n_{k}\big]\Big\}$$
for any $k\geq2.$

\noindent By condition (II) in Theorem \ref{T1}, for any $k\geq1$ and $1\leq j\leq s_{k},$ there exist corresponding words $\xi_{j}^{(k)}$ with $\big|\xi_{j}^{(k)}\big|=\tau$ such that $$\varphi_{s_{k}}^{(k)}\big(\boldsymbol{i}_{1}\boldsymbol{i}_{2}\ldots \boldsymbol{i}_{k}\big)\big[1, n_{k}\big]\xi_{j}^{(k)}\in \mathcal{G}.$$
We now insert the chaotic blocks into the primitive Moran construction. Let $x=H(\boldsymbol{i})$ be the point constructed in Lemma \ref{l3}, and write rite $x=x_{1}x_{2}\ldots x_{k}\ldots$ with $\big|x_{1}x_{2}\ldots x_{k}\big|=q_{m_{k}},$ where the increasing sequence of positive integers $\big\{m_{k}\big\}_{k=1}^{+\infty}$ will be specified later. Using (W)-specification, we glue together the words in We first glue together the words in $J_{k}(\boldsymbol{i})$ to form
$$\begin{aligned}
\chi_{k}=&\varphi_{1}^{(k)}\big(\boldsymbol{i}_{1}\boldsymbol{i}_{2}\ldots \boldsymbol{i}_{k}\big)\big[1, n_{k}\big]\xi_{1}^{(k)}\Join_{m_{k}}\varphi_{2}^{(k)}\big(\boldsymbol{i}_{1}\boldsymbol{i}_{2}\ldots \boldsymbol{i}_{k}\big)\big[1, n_{k}\big]\xi_{2}^{(k)}\Join_{m_{k}}\vee\cdots\\
&\vee\Join_{m_{k}}\varphi_{s_{k}}^{(k)}\big(\boldsymbol{i}_{1}\boldsymbol{i}_{2}\ldots \boldsymbol{i}_{k}\big)\big[1, n_{k}\big]\xi_{s_{k}}^{(k)}
\end{aligned}$$
for any $k\geq1.$ By (W)-specification,  $\chi_{k}\in\mathcal{L}.$ Moreover, by condition (II), for any $\chi_{k}$, there exists there exists $\lambda_{k}$ $|\lambda_{k}|=\tau$ such that $\chi_{k}\lambda_{k}\in\mathcal{G}$, denote this extended word by $\chi_{k}'$.

\noindent To ensure that the inserted chaotic blocks do not reduce the entropy, we choose the insertion positions $m_k$ so sparse that the total length of inserted words is negligible compared with the original Moran length. Specifically, we require
\begin{equation}\label{mkreq}
  \lim\limits_{k\rightarrow+\infty}\frac{\sum\limits_{j=1}^{k+1}2s_{j}
\big(n_{j}+t_{j}+\tau\big)}{l_{m_{k}}}=0,
\end{equation}
Now define the new sequence $y^{k}$  inductively. For each $k\ge 1,$
let $\zeta_{k}(x)$ be the word with $|\zeta_{k}(x)|=\tau$ such that
$x_{k}\zeta_{k}(x)\in\mathcal{G}$, and denote it by $x_{k}'$. Set
$$y^{(1)}:=x_{1}'\Join_{m_{1}}
\chi_{1}'\in\mathcal{L}.$$
By condition (II), there exists $\zeta^{(1)}$ and with length $\tau$ such that
$$y^{(1)}\zeta^{(1)}\Join_{m_2}x_{2}'
\in\mathcal{L}.$$
Furthermore, there exists $\omega_2$ with $|\omega_2|=\tau$ such that
$$y^{(1)}\zeta^{(1)}\Join_{m_2}x_{2}'\omega_2
\in\mathcal{G}.$$
Define
$$y^{(2)}:=y^{(1)}\zeta^{(1)}\Join_{m_2}x_{2}'
\omega_2\Join_{m_2}
\chi_{2}'\in\mathcal{L}.
$$
Proceeding inductively, suppose that $y^{(k)}\in\mathcal{L}$ has been constructed. Then there exist $\zeta^{(k)}$ and $\omega_{k+1}$ with length $\tau$ such that
$$y^{(k)}\zeta^{(k)}\Join_{m_{k+1}}x_{k+1}'\omega_{k+1}
\in\mathcal{G},$$
and we define
$$y^{(k+1)}:=y^{(k)}\zeta^{(k)}\Join_{m_{k+1}}
x_{k+1}'\omega_{k+1}\Join_{m_{k+1}}\chi_{k+1}'\in\mathcal{L}.$$

Since $X_{\mathcal{L}}$ is compact, there exists $y\in X_{\mathcal{L}}$ such that
$$y=\lim\limits_{k\rightarrow\infty}y^{(k)}.$$
We can roughly illustrate this process with the following figure.
\begin{figure}[ptbh]
\centering\includegraphics[width=0.4\textwidth]{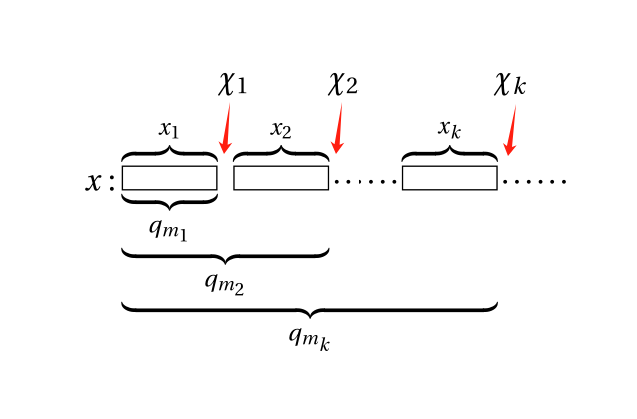} \vspace{-0.3cm}
\label{structure}
\end{figure}
$$\Downarrow$$
\begin{equation}
y: \overbrace{{\begin{tikzpicture}
\draw (0,0) rectangle (1,1/4);
\end{tikzpicture}}}^{x_{1}}\ \;
\underbrace{{\begin{tikzpicture}
\draw (0,0) rectangle (1/4,1/4);
\end{tikzpicture}}}_{\mbox{\tiny$\begin{array}{c}
\textup{connecting}\\
\textup{words}
\end{array}$}}
\overbrace{{\begin{tikzpicture}
\draw (0,0) rectangle (1/2,1/4);
\end{tikzpicture}}}^{\chi_{1}}
\underbrace{{\begin{tikzpicture}
\draw (0,0) rectangle (1/4,1/4);
\end{tikzpicture}}}_{\mbox{\tiny$\begin{array}{c}
\textup{connecting}\\
\textup{words}
\end{array}$}}
\overbrace{{\begin{tikzpicture}
\draw (0,0) rectangle (1,1/4);
\end{tikzpicture}}}^{x_{2}}
\underbrace{{\begin{tikzpicture}
\draw (0,0) rectangle (1/4,1/4);
\end{tikzpicture}}}_{\mbox{\tiny$\begin{array}{c}
\textup{connecting}\\
\textup{words}
\end{array}$}}
\cdots
\underbrace{{\begin{tikzpicture}
\draw (0,0) rectangle (1/4,1/4);
\end{tikzpicture}}}_{\mbox{\tiny$\begin{array}{c}
\textup{connecting}\\
\textup{words}
\end{array}$}}
\overbrace{{\begin{tikzpicture}
\draw (0,0) rectangle (1/2,1/4);
\end{tikzpicture}}}^{\chi_{k-1}}
\underbrace{{\begin{tikzpicture}
\draw (0,0) rectangle (1/4,1/4);
\end{tikzpicture}}}_{\mbox{\tiny$\begin{array}{c}
\textup{connecting}\\
\textup{words}
\end{array}$}}
\overbrace{{\begin{tikzpicture}
\draw (0,0) rectangle (1,1/4);
\end{tikzpicture}}}^{x_{k}}
\cdots
\nonumber
\end{equation}
Let $C$  be the collection of all such limit points $y$ obtained from all $\boldsymbol{i}\in \mathcal{E}.$
We claim that $C$ is still a dynamically defined Moran fractal with full topological entropy. For each $\boldsymbol{i}\in \mathcal{E}$, let $y_{\boldsymbol{i}_{1}\boldsymbol{i}_{2}\ldots \boldsymbol{i}_{k}}$ denote the finite prefix obtained after all chaotic insertions up to the $k$-th level. Set
$$q'_{k}:=|y_{\boldsymbol{i}_{1}\boldsymbol{i}_{2}\ldots \boldsymbol{i}_{k}}|,$$
which is independent of $\boldsymbol{i},$ since the number and lengths of the inserted chaotic blocks are uniform across all $\boldsymbol{i}\in \mathcal{E}.$ Moreover, let $M_{k+1}$ be the total length of all chaotic blocks inserted into the $k+1$-th block during the construction of $y^{(k+1)}$ from $y^{(k)}.$ For any $i_{1}i_{2}\ldots i_{n}\in\mathcal{E}_n,$ we define the basic sets $\Delta'_{i_{1}i_{2}\ldots i_{n}}$ and the length parameters $l'_{n}$ as follows:

\noindent$\bullet$ if $i_{1}i_{2}\ldots i_{n}=\boldsymbol{i}_{1}\boldsymbol{i}_{2}\ldots \boldsymbol{i}_{k}$ for some $k\ge 1,$ set
$$l'_{n}:=q'_{k}\ \textup{and}\ \Delta'_{i_{1}i_{2}\ldots i_{n}}:=\overline{B_{q'_{k}}(y_{\boldsymbol{i}_{1}\boldsymbol{i}_{2}\ldots \boldsymbol{i}_{k}}, \epsilon_{k})}.$$
$\bullet$ Otherwise, if 
$\boldsymbol{i}_{1}\boldsymbol{i}_{2}\ldots \boldsymbol{i}_{k}$ is a prefix of $i_{1}i_{2}\ldots i_{n}$ and 
$i_{1}i_{2}\ldots i_{n}$ is a prefix of $\boldsymbol{i}_{1}\boldsymbol{i}_{2}\ldots \boldsymbol{i}_{k+1},$ set
$$l'_{n}:=q'_{k}+(n-\sum\limits_{j=1}^{k}N_{j})\cdot(n_{k+1}+t_{k+1}+\tau)+M_{k+1}$$
and
$$\Delta'_{i_{1}i_{2}\ldots i_{n}}:=\bigcup_{\mbox{\tiny$\begin{array}{c}
  \boldsymbol{i}_{1}\boldsymbol{i}_{2}\ldots \boldsymbol{i}_{k}\textup{is a prefix of}\ i_{1}i_{2}\ldots i_{n} \textup{and}\\
  i_{1}i_{2}\ldots i_{n}\textup{is a prefix of}\ \boldsymbol{i}_{1}\boldsymbol{i}_{2}\ldots \boldsymbol{i}_{k+1}
  \end{array}$}}\overline{B_{q'_{k+1}}(y_{\boldsymbol{i}_{1}\boldsymbol{i}_{2}\ldots \boldsymbol{i}_{k+1}}, \epsilon_{k+1})},$$
It is straightforward to verify that $C=\bigcap\limits_{n=1}^{+\infty}\bigcup\limits_{{i_{1}i_{2}\ldots i_{n}\in \mathcal{E}_{n}}} \Delta'_{i_{1}i_{2}\ldots i_{n}}$ satisfies the Moran fractal conditions (C1)–(C3). 
Moreover, by \ref{mkreq}, we have 

\begin{equation*}\label{lnjixian}
\lim\limits_{n\rightarrow+\infty}\frac{l'_{n}}{l_{n}}=1.
\end{equation*}
Therefore, applying Proposition \ref{p2}, for any open set $U$ with $U\cap C\neq\emptyset,$ we have


$$h_{top}\big(\sigma, C\cap U\big)=\liminf\limits_{n\rightarrow+\infty}\frac{\log{\#\big(\mathcal{E}_{n}\big)}}{l'_{n}}
\geq\lim\limits_{k\rightarrow+\infty}\frac{\log{c_{k}}}{n_{k}}=h_{top}\big(\sigma\big).$$
Thus $C$ has full topological entropy.

It remains to prove that $C$ is a Xiong chaotic set. Let $A$ be a non-empty subset of $C,$ and let $F: A\rightarrow X_{\mathcal{L}}$ be continuous. First, let $\overline{A}$ be the closure of $A$, since $C$ is compact, so is $\overline{A}$. Moreover, since 
$X_{\mathcal{L}}$ is compact Hausdorff, the continuous map $F: A\rightarrow X_{\mathcal{L}}$ admits a 
unique continuous extension to  $\overline{A}$, for convenience, we still denote it by $F$. Hence, 
$F$ is uniformly continuous on $\overline{A}.$ 

\noindent For each $k\ge 1,$ take a finite open cover of $X_{\mathcal{L}}$ by cylinder sets 
$\{V_{k,1},\ldots,V_{k,r_k}\}$ such that each $\{V_{k,j}\}$ is determined by the first $N_k$ coordinates 
and 
$$\textup{diam}(V_{k,j})<\frac{1}{2{k}},\quad
j=1,\ldots,r_k.$$
By the uniform continuity of $F$ on $\overline{A}$ and the fact that
$$\lim\limits_{n\rightarrow +\infty} \textup{diam}(\Delta'_{i_{1}i_{2}\ldots i_{n}}\cap 
C)\rightarrow0,$$ 
there exists an integer $K_k\ge k$ sufficiently large such that for any $\mu\in \mathcal{E}^{(K_k)}$ with 
$\Delta'_{\mu}\cap \overline{A}\neq\emptyset,$ we have
$$F(\Delta'_{\mu}\cap \overline{A})\subset V_{k,j_k(\mu)}$$
for some $1\le j_k(\mu)\le r_k$.
(Here $\mathcal{E}^{(K)}=\prod\limits_{j=1}^{K}
\{1,\ldots, c_j\}^{N_j}$, and $\Delta'_{\mu}$ denotes the $K_k$-th level basic set indexed by $\mu$.

\noindent
Next, for each cylinder $V_{k,j}$, we choose a representative point $z_{k,j}\in V_{k,j}.$ Since $E_{K_{k}}=\{\mathbf{u}_{1}^{(K_{k})}, \\ \mathbf{u}_{2} ^{(K_{k})}, \ldots, \mathbf{u}_{c_{K_{k}}}^{(K_{k})}\}$ is a maximal $(n_{K_{k}},\epsilon')$-separated set, we may require $n_{K_{k}}$ sufficiently large such that $n_{K_{k}}\ge N_k$ and $\epsilon'<2^{-N_k-1}$, maximality ensures that every cylinder of length $N_k$ contains at least one point of $E_{K_{k}}$. Hence there exists an element $\mathbf{u}_{\ell(k,j)}^{(K_{k})}\in E_{K_{k}}$ such that 
$$\mathbf{u}_{\ell(k,j)}^{(K_{k})}[1,N_k]
=z_{k,j}[1,N_k].$$
Now we define a encoding map
$\Phi_k: \mathcal{E}^{(K_{k})}\rightarrow E_{K_{k}}$ by
\begin{equation*}
\Phi_k(\mu):=
\begin{cases}
\mathbf{u}_{\ell(k,j_k(\mu))}^{(K_{k})} & \text{if} \Delta'_{\mu}\cap \overline{A}\neq\emptyset,\\
\text{any fixed element of } E_{K_k}, & \text{otherwise.}
\end{cases}
\end{equation*}
Since $\mathcal{E}^{(K_{k})}$ is finite, $\Phi_k$ is well-defined and takes only finitely many values.
Recall that in the construction of the chaotic part at level $K_k$, the family
$$\{\varphi_{1}^{(K_{k})}, \varphi_{2}^{(K_{k})}, \ldots, \varphi_{s_{K_{k}}}^{(K_{k})}\},\quad s_{K_{k}}=c_{K_{k}}^{\#(\mathcal{E}^{(K_{k})})}$$
consists of all maps from $\mathcal{E}^{(K_{k})}$ to $E_{K_{k}}$. Therefore, there exists a fixed index $j^{*}\in \{1,\ldots,s_{K_{k}}\}$ such that
$$\varphi_{j^{*}}^{K_{k}}(\mu)
=\Phi_k(\mu),\quad \text{for every}\  \mu\in\mathcal{E}^{(K_{k})}.$$
In the construction of $C$, the chaotic blocks at level $K_k$ are arranged according to a fixed template (the same ordering for all indices $\boldsymbol{i}$), 
the block corresponding to the map $\varphi_{j^{*}}^{(K_{k})}$ has a fixed starting position $p_k$ that is independent of the particular $\boldsymbol{i}$. Since $K_k\rightarrow\infty$, the sequence $\{p_k\}_{k=1}^{+\infty}$ is strictly increasing.

\noindent Finally, fix an $x\in A\subset \overline{A}$. Let $\mu(x)\in\mathcal{E}^{(K_{k})}$ be the $K_k$-th level code of $x$. 
From the choice of $K_k$,
we have 
$F(x)\in V_{k,j_k(\mu(x))}.$
Thus
$$F(x)[1,N_k]=z_{k,j_k(\mu(x))}[1,N_k].$$
By the definition of $\Phi_k$, one has 
$$\Phi_k(\mu(x))[1,N_k]
=\mathbf{u}_{\ell(k,j_k(\mu(x)))}^{(K_{k})}[1,N_k]
=z_{k,j_k(\mu(x))}[1,N_k].$$
By the construction of the chaotic blocks, the point $\sigma^{p_k}(x)$ lies in the block corresponding to $\varphi_{j^{*}}^{(K_{k})}$, consequently,
$$\sigma^{p_k}(x)[1,N_k]=\Phi_k(\mu(x))[1,N_k].$$
Combining the three equalities above, one has 
$$\sigma^{p_k}(x)[1,N_k]=F(x)[1,N_k].$$
Since $N_k\ge k$, the metric on $X_{\mathcal{L}}$ gives
$$d(\sigma^{p_k}(x),F(x))\le 2^{-N_k}<2^{-k}.$$
By letting $k\rightarrow\infty$, we have
$$\lim\limits_{k}\rightarrow\infty\sigma^{p_k}(x)=F(x).$$
Since $x\in A$ is arbitrary, the same sequence $\{p_k\}_{k=1}^{+\infty}$ works for every point in $A$. Therefore, $C$ is Xiong chaotic.
\end{proof}
~\\
\noindent\textbf{Proof of Theorem \ref{T1}} We ensure that the Xiong chaotic set has full topological entropy everywhere by scattering $C$ into each non-empty open subset of $X_{\mathcal{L}}.$ Let
$$\Gamma:=\bigcup\limits_{k=1}^{+\infty}
\bigcup\limits_{\boldsymbol{i}_{1}\boldsymbol{i}_{2}\ldots\boldsymbol{i}_{k}\textup{is}\, \mathcal{E}\textup{-admissible}}\Big\{C\cap \overline{B_{q'_{k}}\big(y_{\boldsymbol{i}_{1}\boldsymbol{i}_{2}\ldots \boldsymbol{i}_{k}}, \epsilon_{k}\big)}\Big\}.$$
Recalling (C3) of Definition \ref{d4}, we can pick up a countable subfamily 
$$\bigcup\limits_{k=1}^{+\infty}\big\{R_{1}^{(k)}, R_{2}^{(k)}\ldots, R_{c_{k}}^{(k)}\big\}\subset\Gamma$$
such that every pair of distinct elements in this subfamily is disjoint, and for $k\geq1,$
$$\big\{R_{1}^{(k)}, R_{2}^{(k)}\ldots, R_{c_{k}}^{(k)}\big\}\subset \bigcup\limits_{\boldsymbol{i}_{1}\boldsymbol{i}_{2}\ldots\boldsymbol{i}_{k}\textup{is}\, \mathcal{E}\textup{-admissible}}\Big\{C\cap \overline{B_{q'_{k}}\big(y_{\boldsymbol{i}_{1}\boldsymbol{i}_{2}\ldots \boldsymbol{i}_{k}}, \epsilon_{k}\big)}\Big\}.$$
Let $R_{h}^{(k)}:=C\cap\overline{B_{q'_{k}}(y_{\boldsymbol{i}_{1}\boldsymbol{i}_{2}\ldots \boldsymbol{i}_{k}}, \epsilon_{k})}$ 
for some fixed integers $k\geq1$ and $1\leq h\leq c_{k}.$  Define a map $\Psi_{h}^{(k)}: R_{h}^{(k)}\rightarrow X_{\mathcal{L}}$ by setting, for any $\alpha\in R_{h}^{(k)},$
$$\Psi_{h}^{(k)}(\alpha)=\mathbf{u}_{h}^{(k)}\big[1, n_{k}\big]\gamma_{h}^{(k)}\Join_{k} \vee \sigma^{n_{k}+t_{k}+\tau}(\alpha_{\boldsymbol{i}_{1}\boldsymbol{i}_{2}\ldots \boldsymbol{i}_{k}}).$$
This yields a well-defined map $\Psi_{h}^{(k)}$ which is injective on $R_{h}^{(k)}$ (since the tail $\Psi_{h}^{(k)}(\alpha)$ uniquely determines $\alpha$). In particular, $\Psi_{h}^{(k)}$ is a topological embedding of $R_h^{(k)}$ into $X_{\mathcal{L}}$, and hence its image inherits the entropy and Moran structure.
 Let 
$$\mathcal{R}:=\bigcup\limits_{k=1}^{+\infty}\bigcup\limits_{h=1}^{c_{k}}R_{h}^{(k)}.$$

We now define a global map $\Psi: \mathcal{R}\rightarrow X_{\mathcal{L}}$ by setting $\Psi\big|_{R_{h}^{(k)}}=\Psi_{h}^{(k)}$ for each pair $(k,h)$. The continuity of $\Psi$ follows directly from (C3) of Definition \ref{d4}. Moreover, the definition of $\Psi_{h}^{(k)}$ ensures that $\Psi(R_{h}^{(k)})$ is contained in $B_{n_{k}}\big(\mathbf{u}_{h}^{(k)}, 3\epsilon_{k}\big).$ Finally, since $\Psi\big(R_{h}^{(k)}\big)$ shares the same dynamically defined Moran structure as   $C\bigcap\overline{B_{q'_{k}}\big(y_{\boldsymbol{i}_{1}\boldsymbol{i}_{2}\ldots \boldsymbol{i}_{k}}, \epsilon_{k}\big)},$ we have  $h_{top} \big( \Psi(R_{h}^{(k)})\big)=h_{top}(\sigma)$.

To verify the entropy condition, fix an arbitrary non-empty open set $U\subset X_{\mathcal{L}}$. By the maximality of the separated sets $E_{k}$, we can find some $k\geq1$, $j\in \big\{1, 2,\ldots ,c_{k}\big\}$ and $\mathbf{u}_{j}^{(k)}\in E_{k}$ such that
$$\Psi\big(R_{j}^{(k)}\big)\subset B_{n_{k}}\big(\mathbf{u}_{j}^{(k)}, 3\epsilon_{k}\big)\subset U.$$
Consequently, the set $C'=\Psi\big(\mathcal{R}\big)$ has full topological entropy in every open set. 

We finally show that $C'=\Psi\big(\mathcal{R}\big)$ is still a Xiong chaotic set. Let $A$ be a non-empty subset of $C'$ and let $F: C'\rightarrow X_{\mathcal{L}}$ be continuous. Then $F\circ\Psi$ is a continuous map from $\Psi^{-1}(A)$ to $X_{\mathcal{L}}.$ Therefore, following the argument in lemma \ref{L4}, we obtain a increasing sequence $\{\ell_k\}_{k=1}^{+\infty}$ of positive integers such that for any $y\in \Psi^{-1}(A)$, we have

$$\lim\limits_{k\rightarrow +\infty}\sigma^{\ell_k}(y)=F\circ\Psi(y).$$
Fix $x=\Psi(y)$ with $y\in \Psi^{-1}(A)$. By the definition of $\Psi$, there exists a constant $L=|\mathbf{u}_{h}^{(k)}[1, n_{k}]\gamma_{h}^{(k)}|$ (independent of 
$y$) such that for all sufficiently large $k$, one has 
 $$\sigma^{p_k}(\Psi(y))=\sigma^{\ell_k}(y),$$
where $p_k:=\ell_k+L-(n_k+t_k+\tau)$ (after discarding finitely many initial terms if necessary to ensure $p_k>0$). Hence $\{p_k\}$ is an increasing sequence of positive integers and 
$$\lim\limits_{k\rightarrow +\infty}\sigma^{p_k}(x)=\lim\limits_{k\rightarrow +\infty}\sigma^{p_k}(\Psi(y))=\lim\limits_{k\rightarrow +\infty}\sigma^{\ell_k}(y)=F(\Psi(y))=F(x).$$

 This completes the proof of Theorem \ref{T1}.\\

\end{document}